 \documentclass[12pt,a4paper]{amsart}

 \usepackage{amsmath}
 \usepackage{amsfonts}
\usepackage{amssymb}

\numberwithin{equation}{section}

\theoremstyle{plain}
\newtheorem{Th}{Theorem}[section]
\newtheorem{Lemma}[Th]{Lemma}

\newtheorem{Prop}[Th]{Proposition}

 \theoremstyle{definition}
\newtheorem{Def}[Th]{Definition}

\newtheorem*{Rem}{Remark}
\newtheorem{?}[Th]{Problem}
\newtheorem{Ex}[Th]{Example}

\newcommand{\sk}[1]{\langle #1 \rangle }

\newcommand \iS {\mathcal{S}}

\newcommand \re {\mathrm{Re}\:}

\newcommand{\setZ}{\ensuremath{\mathbb{Z}}}
\newcommand{\RR}{\ensuremath{\mathbb{R}}}

\newcommand{\1}{{\bf 1}}
\newcommand{\E}{{\bf E}}
\newcommand{\D}{{\bf D}}
\newcommand{\Df}{\overline \Delta}
\newcommand{\df}{\overline \delta}
\renewcommand{\Pr}{{\bf P}}

\begin{document}

\title{Difference sets and positive exponential sums I. General properties}

\author{M\'at\'e Matolcsi}
\address{Alfr\'ed R\'enyi Institute of Mathematics\\
     Budapest, Pf. 127\\
     H-1364 Hungary\\
     (also at BME Department of Analysis, Budapest, H-1111, Egry J. u. 1)
} \email{matomate@renyi.hu}

\thanks{The authors were supported by the ERC-AdG 228005, and OTKA Grant No. K81658, and M.M. also by the Bolyai Scholarship.}

\author{Imre Z. Ruzsa}
\address{Alfr\'ed R\'enyi Institute of Mathematics\\
     Budapest, Pf. 127\\
     H-1364 Hungary
}
\email{ruzsa@renyi.hu}


\begin{abstract}
We describe general connections between intersective properties of sets in Abelian groups and positive exponential sums. In particular, given a set $A$ the maximal size of a set whose difference set avoids $A$ will be related to positive exponential sums using frequencies from $A$.
\end{abstract}

\maketitle

\section{Introduction}\label{secintro}

This work studies the difference-intersective property of sets,
that is, the maximal size (or density) of a set whose difference set avoids a given set. We will explore connections to positive exponential sums using frequencies from the given set.
In this first part we establish some general properties for sets in finite
commutative groups. In the second part
we plan to consider power residues in ${\mathbb {Z}}_m$, and in the third part sets of $k$-th
powers in ${\mathbb {Z}}$.

Difference sets are always symmetric and contain 0; similarly, the
spectrum of a positive exponential sum is symmetric and contains 0. This
motivates the following definition.

\begin{Def} 
Let $G$ be a finite commutative group. We call a set $A\subset G$ a \emph{standard
set}, if $A=-A$ and $0\in A$.
\end{Def}

We found the above version the most comfortable to work with; other
versions are also possible.

\begin{Def}  \label{deltak}
Let $G$ be a finite commutative group, $|G|=q$, and let $A\subset G$ be a
standard set. Write
\begin{equation*} \begin{aligned}
      \Delta (A) &= \max \bigl\{|B|: B\subset G, (B-B)\cap A = \{0\}  \bigr \},  \\
   \overline \Delta (A) & = \max \bigl\{|B|: B\subset G, B-B\subset A \bigr \},
  \end{aligned} \end{equation*}

\[   \delta (A) = \Delta (A)/q,  \]
\[   \overline \delta (A) = 1/\overline \Delta (A).  \]
We call $\delta (A)$ the \emph{measure of intersectivity} of the set $A$.
\end{Def}

Next we list the quantities related to positive character sums. We fix
our notation as follows. A \emph{character} is a homomorphism into
\[   {\mathbb {C}}_1 = \{z\in {\mathbb {C}}: |z|=1\} .  \] The set of all characters is the {\it dual group} of $G$, denoted by $\hat{G}$.
 We will use  additive notation  for $G$ and multiplicative notation for $\hat{G}$, and accordingly $\1\in \hat{G}$ denotes the
identity element of $\hat{G}$, the principal character.
The \emph{Fourier transform} of a function $f$ on $G$ is defined as
\[   \hat f(\gamma ) = \sum _{x\in G} \gamma (x) f(x) . \]

We define certain classes of functions, whose behaviour on $A$ and $G\setminus A$ is prescribed in
various senses. The notation $f\not\equiv 0$ means that $f$ is not identically zero. Put

\begin{equation*} \begin{aligned}
   \mathcal S(A) & = \left \{f: G\to {\mathbb {R}}, f\not\equiv 0, f|_{G\setminus A} = 0 \right \},  \\
   \mathcal S^-(A) & = \left \{f: G\to {\mathbb {R}}, f\not\equiv 0, f|_{G\setminus A} \leq 0 \right \},  \\
   \mathcal S^+(A) & = \left \{f: G\to {\mathbb {R}}, f\not\equiv 0, f|_{G\setminus A} =0, f|_A\geq 0 \right \},  \\
   \mathcal S^\pm (A) & = \left \{f: G\to {\mathbb {R}}, f\not\equiv 0, f|_{G\setminus A} \leq 0, f|_A\geq 0 \right \}.
 \end{aligned} \end{equation*}

These classes of functions are used to define the relevant quantities in relation with positive exponential sums.

\begin{Def}  \label{lambdak}
Let $G$ be a finite commutative group, $|G|=q$, and let $A\subset G$ be a
standard set. Write
\[   \lambda (A) = \min \left  \{\frac{f(0)}{\hat f(\1)}: f\in  \mathcal S(A), \hat f(\gamma )\geq 0 \text{ for all } \gamma \right \}, \]
\[   \lambda ^-(A) = \min \left  \{\frac{f(0)}{\hat f(\1)}: f\in  \mathcal S^-(A), \hat f(\gamma )\geq 0 \text{ for all } \gamma\right \}, \]
\[   \lambda ^+(A) = \min \left  \{\frac{f(0)}{\hat f(\1)}: f\in  \mathcal S^+(A), \hat f(\gamma )\geq 0 \text{ for all } \gamma  \right\}, \]
\[   \lambda ^\pm (A) = \min \left  \{\frac{f(0)}{\hat f(\1)}: f\in  \mathcal S^\pm (A), \hat f(\gamma )\geq 0 \text{ for all }\gamma\right  \}. \]
Sometimes $\lambda (A)$ is called the \emph{Tur\'an constant}, $\lambda ^-(A)$ the
\emph{Delsarte constant} of the set $A$ (for the history of these names and some related problems see \cite{revesz}).
\end{Def}

Of these quantities $\lambda ^\pm $ seems to be the least interesting; we include it
to exhaust all possible combinations of restrictions on $A$ and $G\setminus A$.
Seemingly these definitions depend on the ambient group $G$;  in the next
section we will show that this is not the case, so the notations are justified.

\medskip

We shall study inequalities between these numbers; how they change under
set-theoretical operations (union, intersection, complement, direct product);
and how they behave for a random set.

The main inequality connecting the various $\delta$ and $\lambda$ quantities is the following.

\begin{Th}  \label{basic}
Let $G$ be a finite commutative group, $|G|=q$, and let $A\subset G$ be a
standard set. We have
\begin{equation} \label{foegyenlotlenseg}
     1/q \leq \delta (A) \leq  \lambda ^-(A) \leq  \left\{ \begin{matrix} \lambda (A) \\ \lambda ^\pm (A) \end{matrix}
\right\} \leq  \lambda ^+(A) \leq  \overline \delta (A) \leq  1 .
\end{equation}

All the inequalities can hold with equality, as well as with strict
inequality. There is no inequality between $\lambda (A) $ and $ \lambda ^\pm (A)$; each can be
greater than the other, and they can also be equal.
\end{Th}

We will prove this theorem in Section \ref{secbasicineq}. The main unsolved problem is whether there is any connection between these
quantities in the other direction.

\begin{?}  \label{problem1}
Is there a function $f: [0,1]\to [0,1]$ such that $f(x)\to 0$ as $x\to 0$ and we
have always $\lambda ^-(A)\leq f\bigl( \delta (A) \bigr) $?
Is there such a function for which we have always $\lambda (A)\leq f\bigl( \lambda ^-(A) \bigr) $?
\end{?}

This question can be asked for any other pair of the quantities defined
above. We have the following partial answer.

\begin{Th} \label{norelations}

(a) Let $G$ be a finite commutative group, $|G|=q$, and assume that $3 \nmid  q$. There is a standard set
$A\subset G$ such that $\overline \delta (A)=1/2$ and
\[   \lambda ^+(A) \leq  c q^{-1/6} ( \log q)^{1/2},  \]
with an absolute constant $c$.

(b)  Let $\varepsilon >0$. For every sufficiently large $n$ there is a standard set
$A\subset {\mathbb {Z}}_2^n $ such that
\[   \lambda(A)<\varepsilon , \ \lambda^\pm(A)> 1/2 - \varepsilon  .  \]

(c) Let $\varepsilon >0$. For infinitely many values of $q$ there is a standard set
$A\subset {\mathbb {Z}}_q$ such that
\[   \delta (A)<\varepsilon , \ \lambda^+(A)> 1/2 - \varepsilon  .  \]

\end{Th}

We will prove part (a) of this theorem in Section \ref{secrandom} and part (b) in Section  \ref{secdyadic}.
Part (c) is essentially a theorem of Bourgain \cite{bourgain}
Bourgain's setting and terminology is quite different from ours. We do not give an account of his method
in the hope that the stronger result in part (b) can also be extended to cyclic groups. We also remark here
that the most difficult part in the proof of part (b) is a result of Samorodnitsky \cite{samorodnitsky98}; more details are given in Section \ref{secdyadic}.

\medskip

Most of the defined quantities make sense also in infinite groups; the
exception is $\delta $, whose definition involves division by $q$. Here the proper
generalization involves a concept of density; a very general formulation in locally Abelian groups can be found in a
paper of R\'ev\'esz \cite{revesz}.
Here we restrict our attention to the finite case.

\medskip

It seems to be difficult to say anything nontrivial about the cases of
equality in Theorem \ref{basic}. However, the extremal values are easily described.

\begin{Prop} \label{extremcases}
Let $G$ be a finite commutative group, $|G|=q$, and let $A\subset G$ be a
standard set.

(a) If $A=G$, then
\begin{equation} \label{minimum}
\delta (A) = \lambda ^-(A) = \lambda (A) = \lambda ^\pm (A) = \lambda ^+(A) = \overline \delta (A) = 1/q .
\end{equation}
In any other case $\delta (A)\geq 2/q$.

(b) If $A=\{0\}$, then
\begin{equation} \label{maximum}
\delta (A) = \lambda ^-(A) = \lambda (A) = \lambda ^\pm (A) = \lambda ^+(A) = \overline \delta (A) = 1 .
\end{equation}
In any other case $\overline \delta (A)\leq 1/2$.

\end{Prop}

Both statements are immediate consequences of the definitions.

\section{Invariance properties}\label{secinvariance}

In Definition \ref{deltak} and \ref{lambdak} the ambient group $G$ occurs. A
set $A$ may be a subset of several groups (they being subgroups of a common
group), and the definitions could, in principle, return
different values. We show here that this is not the case, hence our notations $\delta (A)$,
$\lambda (A)$, etc. are justified.

\medskip

To formulate the results rigorously we temporarily extend the notation and write
$\delta (A, G)$, $\lambda (A, G), \dots $,  instead. Also, it will be convenient to introduce the following general notation.

\begin{Def}\label{restriction}
If $X$ is a subset of $Y$, and $f: Y\to \RR$ is a function on $Y$ then $f_X$ denotes the restriction of $f$ to $X$. Conversely, if $g: X\to \RR$ is a function on $X$ then $g^Y$ denotes the extension of $g$ to $Y$ with value 0 outside $X$.
\end{Def}

\begin{Th} \label{ambient}
Let $G$ be a commutative group, $G_1, G_2$ finite subgroups of $G$, and
$A\subset G_1\cap G_2$ a standard set. Let $\varphi $ be any of the functionals $\delta , \overline \delta , \lambda , \lambda^-,
\lambda^+, \lambda^\pm $. We have
\[   \varphi (A, G_1) = \varphi  (A, G_2).  \]
\end{Th}

\begin{proof}
The claim is obvious for $ \overline \delta$, whose definition does not contain any reference
to $G$. We prove the rest.

\medskip

First we consider the particular case when $G_2=G$. Write $|G_1|=q_1$, $|G|=q$.

\medskip

Consider the case of $\delta$. Let $B, B_1$ be the maximal sets in $G$ and
$G_1$, resp., with the property that
\[ (B-B) \cap A =(B_1- B_1) \cap A = \{0\} . \]
Consider a coset $t+G_1$ of $G_1$. Since the set
$ B_t = (t+G_1) \cap  B $ satisfies $B_t'=B_t-t\subset G_1$ and $(B_t'-B_t')\cap A\subset \{0\}$, we conclude
$|B_t|\leq |B_1|$. Applying this for each coset and summing we obtain $|B|\leq (q/q_1) |B_1|$.
On the other hand, take a representative from each coset, say $t_1, \ldots, t_{q/q_1}$.
The set $ \bigcup (t_i+B_1)$ demonstrates $|B| \geq   (q/q_1) |B_1|$ .

\medskip

Consider now the case when $\varphi $ is any of the functionals $\lambda , \lambda^-,
\lambda^+, \lambda^\pm$. First, if $f: G_1 \to \RR$ is an appropriate function with ${f(0)}/{\hat{f}(0)}=\varphi (A, G_1)$ then it is straightforward to see that $f^G$ has all the required properties to testify that $\varphi (A, G)\leq \varphi (A, G_1)$.

To see the reverse inequality assume that $g: G\to \RR$ is an appropriate function with ${g(0)}/{\hat{g}(\1)}=\varphi (A, G)$,
and consider the restricted function $h=g_{G_1}$. If $\varphi=\lambda$ or $\lambda^+$ then $h$ obviously testifies that $\varphi (A, G_1)\leq \varphi (A, G)$. In the case $\varphi=\lambda^-$ or  $\lambda^{\pm}$ we still have $h(0)=g(0)$ and $\hat{h}(\1)\leq \hat{g}(\1)$, and therefore ${h(0)}/{\hat{h}(\1)}\leq \varphi (A, G)$.
Also, $h$ falls into the class $\iS^- (A,G_1)$ or $\iS^\pm (A,G_1)$. It remains to show that the Fourier coefficients of $h$ are nonnegative.
To see this, let $\gamma\in \hat{G_1}$ and consider all $\psi\in \hat{G}$ such that $\psi_{G_1}=\gamma$. There exist ${q}/{q_1}$ such characters $\psi$. Then
\begin{equation}\label{sgtrick} \begin{aligned}
0\leq \sum_{\psi: \psi_{G_1}=\gamma} \hat{g} (\psi) & =\sum_{\psi} \sum_{x\in G} \psi(x)g(x)=\sum_{\psi}\sum_{x\in G_1}\psi(x)g(x)+ \sum_{\psi}\sum_{x\notin G_1}\psi(x)g(x) \\
& = \frac{q}{q_1}\hat{h}(\gamma)+ \sum_{x\notin G_1}\left( g(x)\sum_{\psi}\psi(x) \right)=\frac{q}{q_1}\hat{h}(\gamma)
\end{aligned} \end{equation}
where we have used that the inner summation in the last sum always returns 0. This shows that $\hat{h}(\gamma)\geq 0$.

\medskip

Finally, in the general case, $G_1, G_2\leq G$, let $H\leq G$ be the subgroup generated by $G_1$ and $G_2$. Then $H$ is also finite, and by the argument above $\varphi (A, G_1) = \varphi (A, H)= \varphi  (A, G_2)$.
\end{proof}

\section{The basic inequality}\label{secbasicineq}

In this section we prove Theorem \ref{basic}. We will only prove $\delta(A)\leq \lambda^-(A)$ and $\lambda^+(A)\leq \overline \delta (A)$, the other inequalities are trivial.

\medskip

To see $\delta(A)\leq \lambda^-(A)$, assume $f\in\iS^- (A)$ is any function such that $\hat{f}\geq 0$, and $B\subset G$ is such that $(B-B)\cap A =\{0\}$. Introduce the function $\hat{B}(\gamma)=\sum_{b\in B}\gamma(b)$, and notice that $|\hat{B}(\gamma)|^2=\sum_{b_1,b_2\in B} \gamma(b_1-b_2)$.
We now evaluate the sum $S=\sum_{\gamma\in \hat{G}} \hat{f}(\gamma)|\hat{B}(\gamma)|^2$. On the one hand, all terms are nonnegative, hence by considering the term $\gamma=\1$ only we get a lower bound $S\geq \hat{f}(\1)|B|^2$.
On the other hand, by exchanging the order of summation and using the Fourier inversion formula we obtain
$$ S=\sum_{\gamma}\sum_{b_1, b_2} \hat{f}(\gamma) \gamma(b_1-b_2)=\sum_{b_1, b_2} \sum_{\gamma} \hat{f}(\gamma) \gamma(b_1-b_2)=q\sum_{b1, b2} f(b_1-b_2).$$
In the last summation all the terms are non-positive by assumption, except when $b_1=b_2$. Hence, $S\leq f(0)|B|$, and comparing the lower and upper bounds $\frac{|B|}{q}\leq \frac{f(0)}{\hat{f}(\1)}$ follows.

To see $\lambda^+(A)\leq \overline \delta (A)$, assume $B\subset G$ is such that $B-B\subset A$. Define the function $f:G\to \RR$ by setting $f(x)$ to be the number of ways $x$ can be written in the form $x=b_1-b_2$ where $b_1, b_2\in B$. In other words, $f=1_B\ast 1_{-B}$. Clearly, $f\in \iS^+ (A)$ and
$$\frac{f(0)}{\hat{f}(0)}=\frac{|B|}{|B|^2}=\frac{1}{|B|}. $$
 Furthermore, $\hat{f}=|\hat{1_B}|^2\geq 0$, so $f$ satisfies each criterion in the definition of $\lambda^+(A)$, and we conclude that
$\lambda^+(A)\leq {1}/{|B|}$.

\newcommand{\B}{\overline B}
\begin{Ex}
  The cases when all our quantities are equal are connected with tilings. Indeed, assume that $\delta(A)=\overline \delta(A)=\delta$, say.
  Take sets $B, \B$ such that
  \begin{align*}
    |B|& = \delta q, & (B-B)& \cap A = \{0\}, \\
 |\B| & = 1/\delta, & (\B-\B) & \subset  A .
  \end{align*}
  The conditions on difference sets imply that all the sums $x+y: x\in B, y\in\B$ are distinct and their number is
$  |B|  |\B| =q$, so $(B, \B)$ is a tiling of $G$. Conversely, any tiling  $(B, \B)$ induces examples of equality as follows.
Take any set $E \subset G \setminus \bigl( (B-B) \cup (\B-\B) \bigr)$. The set
$A=(\B-\B) \cup E $ satisfies $\overline \delta(A) \leq 1/  |\B|$ and $\delta(A) \geq |B-B|/q  = 1/  |\B|$, hence
$\delta(A)=\overline \delta(A) = 1/  |\B|$.
\end{Ex}

\begin{Ex}
  Let $q$ be a prime, $q \equiv 1 \pmod 4 $, $G=\setZ_q$ and let $A$ be the set of quadratic residues. By the familiar propertes of Gaussian
sums one easily shows that $\lambda^-(A) = \lambda^+(A) = 1/\sqrt{q}$ (the case of composite $q$ is more difficult). On the other hand
$\delta(A)< 1/\sqrt{q} < \df(A)$, since the $\delta$'s must be rational. It is natural to conjecture that $\delta$ is much smaller, perhaps
of size $O\bigl( ( \log q)^c \bigr)$, like for a random set (for random sets see Section \ref{secrandom}), but nothing much stronger than
$ 1/\sqrt{q}$ is known.
\end{Ex}

Examples where the  $\lambda$'s are different, as well as examples where the $\delta$'s are very different from the $\lambda$'s, will be given
in Sections \ref{secrandom} and \ref{secdyadic}.

\section{Complements and linear duality}\label{secduality}

\begin{Def}
Two standard sets in a group $G$ are \emph{standard complements}, if
$A\cup A'=G$ and $A\cap A'=\{0\}$.
\end{Def}

The various quantities $\delta$ and $\lambda$ of standard complements are nicely related to each other by the following theorem.

\begin{Th}  \label{duality}
Let $G$ be a finite commutative group, $|G|=q$, and let $A, A'\subset G$ be
standard complements. We have
\begin{equation} \label{dualitas}
\delta (A) \overline \delta (A') = \lambda (A) \lambda (A') = \lambda^-(A) \lambda^+(A') = \lambda^\pm (A) \lambda^\pm (A') =1/q.
\end{equation}
\end{Th}

We express this by saying that $\delta $ and $\overline \delta $ are dual quantities, and so are $ \lambda^-$ and $ \lambda^+$,
while $\lambda $ and $ \lambda^\pm $ are self-dual.

\begin{proof}
The relation $\delta (A) \overline \delta (A')=1/q$ is clear from $\overline{\Delta}(A')=\Delta (A)$.

We prove the other three equalities.
 Let $\varphi$ denote one of the functionals $\lambda, \lambda^-, \lambda^{\pm}$ and $\varphi'$ its dual, i.e. $\lambda, \lambda^+, \lambda^{\pm}$, respectively.

First we show the easy inequality ${1/q}\leq \varphi(A)\varphi'(A')$. To this end take any two functions $f_1$ and $f_2$ satisfying the requirements
in the definition of $\varphi(A)$ and $\varphi'(A')$. Consider the function $h=f_1 f_2$. Then $h(0)=f_1(0) f_2(0)$ and
$$\hat{h}(\1)=\frac{1}{q}(\hat f_1\ast \hat f_2)(\1)\geq \frac{1}{q}(\hat f_1 (\1) \hat f_2 (\1)).$$ Also, by the signs of $f_1$ and $f_2$ we see that $h$ is non-positive everywhere except at 0. Therefore $h(0)\geq \hat h(\1)$ which implies
$$f_1(0) f_2(0)\geq \frac{1}{q}(\hat f_1 (\1) \hat f_2 (\1)). $$

To prove the converse inequality we will apply linear duality.
 Let $f$ be any real function on $G$ and consider the values $f(x)$ as variables (as $x$ ranges through $G$). Consider the following systems of inequalities:

\smallskip

\noindent For $\varphi=\lambda$:
\begin{equation}\label{duallambda}
f(x)=0 \ \ {\rm if} \ \  x\notin A, \ \ \sum_{x\in G} f(x)\geq 1,  \ \ \sum_{x\in G}f(x)\gamma(x)\geq 0 \ \ \ {\rm if} \ \ \1\neq\gamma\in \hat{G}
\end{equation}

\noindent For $\varphi=\lambda^-$:
\begin{equation}\label{duallambdaminus}
f(x)\leq 0 \ \ {\rm if} \ \ x\notin A, \ \ \sum_{x\in G} f(x)\geq 1,  \ \ \sum_{x\in G}f(x)\gamma(x)\geq 0 \ \ \ {\rm if} \ \ \1\neq\gamma\in \hat{G}
\end{equation}

\noindent For $\varphi=\lambda^\pm$:
\begin{equation}\label{duallambdaplusminus}
f(x)\leq 0 \ {\rm if} \  x\notin A, \ f(x)\geq 0 \ {\rm if} \ x\in A, \ \sum_{x\in G} f(x)\geq 1, \ \sum_{x\in G}f(x)\gamma(x)\geq 0 \ {\rm if} \ \1\neq\gamma\in \hat{G}
\end{equation}

In each case we know that the inequalities imply $f(0)\geq \varphi(A)$. Therefore, by the principle of linear duality (see e.g. \cite{vanderbei} Theorem 5.2 for a convenient formulation), the inequality $f(0)\geq \varphi(A)$ is the weighted linear combination of the inequalities above, i.e. there exist coefficients $h_1(\1)\geq 0$, $h_1(\gamma)\geq 0$ (for $\gamma\neq \1$), and $h_2(x)$ (with appropriate signs for $x\in A$ and $x\notin A$; see the restrictions below), such that

\begin{equation}\label{dual4} \begin{split}
f(0) & =h_1(\1)\left ( \sum_{x\in G} f(x)\right )+\sum_{\gamma\neq 0}h_1(\gamma)\left ( \sum_{x\in G} f(x)\gamma(x)\right ) + \sum_{x\in G}h_2(x)f(x) \\
 & \geq h_1(\1)=\varphi(A).
\end{split} \end{equation}

The restrictions for $h_2(x)$ are as follows:

\smallskip

\noindent For $\varphi=\lambda$:
\begin{equation}\label{restlambda}
h_2(x)=0 \ \ {\rm if} \ \ x\in A
\end{equation}
For $\varphi=\lambda^-$:
\begin{equation}\label{restlambdaminus}
h_2(x)= 0 \ \ {\rm if} \ \ x\in A,  \ h_2(x)\leq 0 \ \ {\rm if} \ \ x\notin A
\end{equation}
For $\varphi=\lambda^\pm$:
\begin{equation}\label{restlambdaplusminus}
\ h_2(x)\geq 0 \ \ {\rm if} \ \ x\in A,  \ h_2(x)\leq 0 \ \ {\rm if} \ \ x\notin A \end{equation}

From \eqref{dual4} we conclude that $h_1(\1)=\varphi(A)$. Let $g:G\to \RR$ be the function such that $\hat{g}=h_1$. Then $\hat{g}\geq 0$ by definition. Also, $\hat{g}(\1)=\varphi(A)$, and $q g(0)=\sum_{\gamma\in \hat{G}}h_1(\gamma)=1-h_2(0)$, as it is the coefficient of $f(0)$ in \eqref{dual4}.  For any $x\neq 0$, comparing the coefficients of $f(x)$ in \eqref{dual4} we get
$$0=\sum_{\gamma\in \hat{G}}h_1(\gamma)\gamma(x)+h_2(x)=q g(x)+h_2(x), $$
 which implies the following inequalities:

\smallskip

\noindent For $\varphi=\lambda$:
\begin{equation}\label{gxlambda}
g(x)=0 \ \ {\rm if} \ \ x\in A \ (x\neq 0)  \Rightarrow\  g \in \iS(A').
\end{equation}
For $\varphi=\lambda^-$:
\begin{equation}\label{gxlambdaminus}
g(x)= 0 \ \ {\rm if} \ \ x\in A \ (x\neq 0),  \ g(x)\geq 0 \ \ {\rm if} \ \ x\notin A   \Rightarrow\   g \in \iS^+(A').
\end{equation}
For $\varphi=\lambda^\pm$:
\begin{equation}\label{gxlambdaplusminus}
g(x)\leq 0 \ \ {\rm if} \ \ x\in A \ (x\neq 0),  \ g(x)\geq 0 \ \ {\rm if} \ \ x\notin A     \Rightarrow\  g \in \iS^\pm(A').
\end{equation}

Therefore, the function $g$ testifies that
$$\varphi'(A')\leq \frac{1-h_2(0)}{q\varphi(A)}\leq \frac{1}{q\varphi(A)} . $$



\end{proof}

\begin{Rem}
  Perhaps the first application of linear duality to this sort of problem is in a paper by the second author\cite{r84a}; a
good account can be found in Montgomery's book\cite{montgomery94}.
\end{Rem}

\section{Automorphisms}\label{secautomorphisms}

In this section we state some simple but useful properties of the behaviour of our quantities under automorphisms.

\begin{Prop} \label{automorph1}
Let $G$ be a finite commutative group, $\pi$ an automorphism of $G$ and let
$\varphi $ be any of the functionals $\delta , \overline \delta , \lambda , \lambda^-, \lambda^+, \lambda^\pm $. For every $A\subset G$ we have
\[   \varphi (A) = \varphi  (\pi(A)).  \]
\end{Prop}

We omit the simple proof. As an application, let $q$ be a prime, $q\equiv1 \pmod 4$, $G= \setZ_q$, and let $A$ be the set of quadratic residues.
The standard complement of $A$ is $A'$, the set of nonresidues. Since the multiplication by a nonresidue is an automorphism that
transforms $A$ into $A'$, we have $ \varphi (A) = \varphi(A')$ for any of the above functionals. On the other hand,
from Theorem \ref{duality} we know that $\lambda(A)\lambda(A')=  \lambda^\pm(A) \lambda^\pm(A')=1/q$, so we immediately get that
$\lambda(A)=  \lambda^\pm(A) =1/\sqrt{q}$. While this fact, and also the values of $\lambda^+(A)$ and $\lambda^-(A)$ are easily found directly using
Gaussian sums, it is somewhat surprising that we can find them without resorting to any real number theory. Unfortunately this
argument does not work for composite moduli or higher powers.

\begin{Prop} \label{automorph2}
Let $G$ be a finite commutative group, $A\subset G$, and let $\Pi$ be the set of those automorphisms that leave $A$ fixed (as a set,
not necessarily pointwise). Let $\varphi $ be any of the functionals $ \lambda , \lambda^-, \lambda^+, \lambda^\pm $, and let $\mathcal T$ be the corresponding
class of functions (one of $\mathcal S(A), \mathcal S^-(A), \mathcal S^+(A)$ or $ \mathcal S^\pm(A)$, restricted to functions with
nonnegative Fourier transform).
There is an $f\in\mathcal T$ such that
$\varphi(A)= f(0)/\hat f(\1)$ which is invariant under $\Pi$, that is, $f=f \circ\pi$ for all $\pi\in\Pi$.
\end{Prop}

\begin{proof}
  Indeed, take any $f_0\in\mathcal T$ for which $\varphi(A)= f_0(0)/\hat f_0(\1)$    and form
   \[ f(x) = \sum_{\pi\in\Pi} f(\pi(x)) .\]
\end{proof}

For sets that have lots of automorphisms, like power residues, this restricts the class of functions to be considered for
finding the value of $\lambda$, etc.

\section{Union and intersection}\label{secunion}

In this section we consider the behaviour of the various $\delta$ and $\lambda$ quantities under intersection and union of standard sets.

\begin{Th}  \label{deltametszet}
Let $G$ be a finite commutative group, $|G|=q$, and let $A_1, A_2\subset G$ be
standard sets.  We have
\begin{equation} \label{dm}
\overline \delta (A_1 \cap A_2) \leq q \overline \delta (A_1)\overline \delta( A_2).     \end{equation}
\end{Th}

\begin{proof}
Take sets $B_i$ such that $B_i-B_i\subset A_i$, $i=1,2$. Any set of the form
$B=B_1\cap(t-B_2)$ satisfies $B-B\subset A_1\cap A_2$, and an obvious averaging argument shows that there exists a $t$ such that
$|B| \geq |B_1| |B_2|/q$.
\end{proof}

\begin{Th}  \label{deltaunio}
Let $G$ be a finite commutative group, $|G|=q$, and let $A_1, A_2\subset G$ be
standard sets.  We have
\begin{equation} \label{du}
\delta (A_1 \cup A_2) \geq  \delta (A_1)\delta( A_2),     \end{equation}
\end{Th}
\begin{proof}
Using the duality $\delta (A) \overline \delta (A')=1/q$ the statement follows from the previous result applied to the standard complements of $A_1$ and $A_2$.
\end{proof}

\begin{Th}  \label{lambdametszet}
Let $G$ be a finite commutative group, $|G|=q$, and let $A_1, A_2\subset G$ be
standard sets.  We have
\begin{equation} \label{lm}
\lambda (A_1 \cap A_2) \leq q  \lambda (A_1)\lambda( A_2),
\end{equation}
\begin{equation} \label{lmp}
\lambda^+ (A_1 \cap A_2)  \leq q  \lambda^+ (A_1)\lambda^+( A_2),
\end{equation}
\begin{equation} \label{lmm}
\lambda^- (A_1 \cap A_2)  \leq q   \lambda^- (A_1)\lambda^+( A_2),
\end{equation}
\begin{equation} \label{lmpm}
\lambda^\pm (A_1 \cap A_2)  \leq q   \lambda^\pm (A_1)\lambda^+( A_2).
\end{equation}
\end{Th}

\begin{proof}
Let $f_1, f_2$ be functions, belonging to some of the $\iS$-classes of the sets $A_1, A_2$. Their product $h = f_1 f_2$ belongs
to an  $\iS$-class of the intersection as follows:
\begin{alignat*}{3}
  f_1 & \in \iS(A_1),\  & f_2 &\in \iS(A_2) &\ \Rightarrow\  & h \in \iS(A_1 \cap A_2), \\
  f_1 & \in \iS^+(A_1),\  & f_2 &\in \iS^+(A_2) &\ \Rightarrow\  & h \in \iS^+(A_1 \cap A_2), \\
  f_1 & \in \iS^-(A_1),\  & f_2 &\in \iS^+(A_2) &\ \Rightarrow\  & h \in \iS^-(A_1 \cap A_2), \\
  f_1 & \in \iS^\pm (A_1),\  & f_2 &\in \iS^+(A_2) &\ \Rightarrow\  & h \in \iS^\pm (A_1 \cap A_2).
\end{alignat*}

Clearly   $h(0)=f_1(0)f_2(0)$.  Furthermore we have $\hat h=(\hat f_1 \ast \hat f_2)/q$,
which shows that $\hat h\geq 0 $ and
 $\hat h(\1)\geq \hat f_1(\1) \hat f_2(\1)/q$, and we conclude
$$\frac{h(0)}{\hat h(\1)}\leq q \frac{f_1(0)}{\hat f_1(\1)}\frac{f_2(0)}{\hat f_2(\1)}. $$
By taking the minimum over all admissible $f_1, f_2$ we get the inequalities of the theorem.
\end{proof}

\begin{Th}  \label{lambdaunio}
Let $G$ be a finite commutative group, $|G|=q$, and let $A_1, A_2\subset G$ be
standard sets.  We have
\begin{equation} \label{lu}
\lambda (A_1 \cup A_2) \geq  \lambda (A_1)\lambda( A_2),
\end{equation}
\begin{equation} \label{lup}
\lambda^+ (A_1 \cup A_2) \geq  \lambda^+ (A_1)\lambda^-( A_2),
\end{equation}
\begin{equation} \label{lum}
\lambda^- (A_1 \cup A_2) \geq  \lambda^- (A_1)\lambda^-( A_2),
\end{equation}
\begin{equation} \label{lupm}
\lambda^\pm (A_1 \cup A_2) \geq  \lambda^\pm (A_1)\lambda^-( A_2).
\end{equation}
\end{Th}

\begin{proof}
Using the duality relations these statements are easily seen to be equivalent to the statements of the previous theorem applied to the standard complements of $A_1$ and $A_2$. For example, in the case of \eqref{lup} the calculation runs as follows:
\begin{equation*}
\frac{1/q}{\lambda^+(A_1\cup A_2)}=\lambda^-(A_1'\cap A_2')\leq q \lambda^-(A_1') \lambda^+(A_2')=q \frac{1/q}{\lambda^+(A_1)}\frac{1/q}{\lambda^-(A_2)}
\end{equation*}
\end{proof}

Most of the above functionals satisfy a trivial monotonicity property. Let $\varphi $ be any of the functionals $\delta , \overline \delta , \lambda , \lambda^-, \lambda^+ $.
\begin{equation}\label{mon}
\rm{If} \ A_1\subset A_2 \ \rm{then} \ \varphi(A_2)\leq \varphi (A_1).
\end{equation}

This observation can be applied to complement the upper estimates for intersection by the lower estimate
 \[ \varphi(A_1 \cap A_2) \geq \max \bigl( \varphi(A_1), \varphi( A_2) \bigr),  \]
and the lower estimates for union by the upper estimate
 \[ \varphi(A_1 \cup A_2) \leq \min \bigl( \varphi(A_1), \varphi( A_2) \bigr).  \]
Equality holds when $A_1=A_2$, so in general nothing stronger can be asserted.

We will see in Example \ref{exlambdapm} that  inequality \eqref{mon} may fail for $ \lambda^\pm $.

\begin{?}
Find a nontrivial lower estimate for $\lambda^\pm (A_1 \cap A_2)$ and a  nontrivial upper estimate for $\lambda^\pm (A_1 \cup A_2)$.
\end{?}

\section{Subgroups and factor groups}\label{secfactor}

Let $G$ be a commutative group and $H$ a subgroup. We use $G/H$ to
denote the factor group, and we use the cosets of $H$ to represent its
elements. We also introduce the following natural notions.

\begin{Def}
For any set $A\subset G$ we write  $A/H = \{H+a: a\in A\}$
to denote the collection of cosets that intersect $A$ (= the image of $A$
under the canonical homomorphism from $G$ to $G/H$). For any function $f: G\to \RR$ we introduce the factorization of $f$ by $H$ as the function
$f_{/H}$ on $G/H$ defined by $f_{/H}(x+H)=\sum_{t\in H} f(x+t)$. Conversely, for a function $g: G/H \to \RR$ we introduce the lifting $g^{\times H}$ of $g$ to the group $G$ as $g^{\times H}(x)=g(x+H)$.
\end{Def}

\medskip

The following is essentially a result of Kolountzakis and R\'ev\'esz \cite{KR}.

\begin{Th}  \label{factor}
Let $G$ be a finite commutative group, $H$ a subgroup, $G_1=G/H$.
Let $A\subset G$ be a standard set, and put $A_H=A\cap H\subset H$, $A_1=A/H\subset G_1$. We have
\begin{equation}\label{factd}
\delta (A) \geq  \delta (A_H) \delta (A_1),
\end{equation}
\begin{equation}  \label{factoverd}
\overline \delta (A) \geq  \overline \delta (A_H) \overline \delta (A_1),
\end{equation}
\begin{equation}  \label{factl}
\lambda (A) \geq  \lambda (A_H) \lambda (A_1),
\end{equation}
\begin{equation}  \label{factlp}
\lambda^+(A) \geq  \lambda^+(A_H) \lambda^+(A_1),
\end{equation}
\begin{equation}  \label{factlm}
\lambda^-(A) \geq  \lambda^-(A_H) \lambda^-(A_1),
\end{equation}
\begin{equation}  \label{factlpm}
\lambda^\pm (A) \geq  \lambda^\pm (A_H) \lambda^-(A_1).
\end{equation}
\end{Th}

\begin{proof}
To see \eqref{factd} let $B_H$ be a set such that $B_H\subset H$ and $(B_H-B_H)\cap A_H=\{0\}$, and let $B_1\subset G_1$ be a set such that $(B_1-B_1) \cap A_1=\{0\}$. The elements of $B_1$ are cosets of $H$. Take a representative $x_i\in G$ from each such coset, and consider the set $B=\cup_i (x_i+B_H)\subset G$. It is clear that $|B|=|B_H||B_1|$ and $(B-B)\cap A =\{0\}$.

\medskip

Inequality \eqref{factoverd} is equivalent to $\overline \Delta (A)\leq \overline \Delta (A_H) \overline \Delta (A_1)$. Take a set $B\in G$ such that $B-B \subset A$. In each coset $x+H$ there can be at most $\overline \Delta (A_H)$ elements of $B$. Also, the number of cosets that contain some elements of $B$ is at most $\overline \Delta (A_1)$. Therefore, $|B|\leq \overline \Delta (A_H) \overline \Delta (A_1)$.

We will prove the remaining four inequalities. Let $f: G\to \RR$ be any function and consider the functions $f_H : H\to \RR$ and $f_{/H}: G_1\to \RR$. The following implications are straightforward:

\begin{alignat*}{3}
  f & \in \iS_G(A) &\ \Rightarrow\ & f_H \in \iS_H(A_H),  & \ f_{/H} \in \iS_{G_1}(A/H), \\
  f & \in \iS^+(A) &\ \Rightarrow\ & f_H \in \iS^+_H(A_H),  & \ f_{/H} \in \iS^+_{G_1}(A/H), \\
  f & \in \iS^-(A) &\ \Rightarrow\ & f_H \in \iS^-_H(A_H),  & \ f_{/H} \in \iS^-_{G_1}(A/H), \\
  f & \in \iS^\pm (A) &\ \Rightarrow\ & f_H \in \iS^{\pm}_H(A_H),  & \ f_{/H} \in \iS^-_{G_1}(A/H),.
\end{alignat*}

Assuming that $\hat f\geq 0$  the relation $\hat f_H\geq 0$ can be seen in the same manner as in \eqref{sgtrick} in the proof of Theorem \ref{ambient}. Note also that
\begin{equation}\label{e1}
\frac{f_H(0)}{\hat f_H (\1)}=\frac{f(0)}{\sum_{x\in H}f(x)}.
\end{equation}
Furthermore, $\hat f_{/H}\geq 0$ also holds, because for each $\gamma\in \hat G_1$ we have \\
$\hat f_{/H}(\gamma) =\sum_{x+H\in G_1}f_{/H}(x+H)\gamma (x+H)= \sum_{x+H\in G_1} (\sum_{y\in (x+H)} f(y)) \gamma(x+H)=\sum_{x+H\in G_1} (\sum_{y\in (x+H)} f(y) \gamma^{\times H}(y))=\hat f(\gamma^{\times H})\geq 0$. Observing that
\begin{equation}\label{e2}
\frac{f_{/H}(0)}{\hat f_{/H} (\1)}=\frac{\sum_{x\in H}f(x)}{\hat f(\1)}
\end{equation}
and using \eqref{e1} we obtain the required inequalities \eqref{factl}, \eqref{factlp}, \eqref{factlm}, \eqref{factlpm}.
\end{proof}

We note here that the last inequality is less symmetric than the others. We do
not know whether the stronger inequality
\[   \lambda^\pm (A) \geq  \lambda^\pm (A_H) \lambda^\pm (A_1)  \]
holds or not.

\section{Direct products}\label{secdirect}

In this section we consider the behaviour of the various $\delta$ and $\lambda$ quantities under the direct product operation.

\medskip

\begin{Th}
Let $G=G_1\times G_2$ be the direct product of two finite commutative groups, and let $A=A_1\times A_2$, where $A_1\subset G_1$, $A_2\subset G_2$. We have

\begin{equation}\label{dirl}
\lambda (A)=\lambda (A_1) \lambda (A_2),
\end{equation}

\begin{equation}\label{dirlp}
\lambda^+ (A)=\lambda^+ (A_1) \lambda^+ (A_2),
\end{equation}

\begin{equation}\label{dirlm}
\lambda^- (A_1) \lambda^- (A_2)\leq \lambda^- (A)\leq \lambda^- (A_1) \lambda^+ (A_2),
\end{equation}

\begin{equation}\label{dirlpm}
\lambda^\pm (A_1) \lambda^- (A_2)\leq \lambda^\pm (A)\leq \lambda^\pm (A_1) \lambda^+ (A_2).
\end{equation}
\end{Th}

\medskip

\begin{proof}
The claimed lower bounds on $\lambda (A)$,   $\lambda^+ (A)$, $\lambda^- (A)$, $\lambda^\pm (A)$ follow from inequalities \eqref{factl}, \eqref{factlp}, \eqref{factlm}, \eqref{factlpm}, respectively.

To prove the upper bounds, let $f_1$ and $f_2$ be appropriate functions for the sets $A_1$, $A_2$, and consider the function $h(x,y)=f_1(x)f_2(y)$.  The following implications are straightforward:
\begin{alignat*}{3}
  f_1 & \in \iS(A_1),\  & f_2 &\in \iS(A_2) &\ \Rightarrow\  & h \in \iS(A_1 \times A_2), \\
  f_1 & \in \iS^+(A_1),\  & f_2 &\in \iS^+(A_2) &\ \Rightarrow\  & h \in \iS^+(A_1 \times A_2), \\
  f_1 & \in \iS^-(A_1),\  & f_2 &\in \iS^+(A_2) &\ \Rightarrow\  & h \in \iS^-(A_1 \times A_2), \\
  f_1 & \in \iS^\pm (A_1),\  & f_2 &\in \iS^+(A_2) &\ \Rightarrow\  & h \in \iS^\pm (A_1 \times A_2).
\end{alignat*}
Also, $\hat h\geq 0$ follows from $\hat f_1\geq 0$ and $\hat f_2\geq 0$, and $h(0)=f_1(0)f_2(0)$ and $\hat h(\1)=\hat f_1(\1) \hat f_2(\1)$. Therefore, the function $h$
testifies the upper bounds in \eqref{dirl}, \eqref{dirlp}, \eqref{dirlm} and \eqref{dirlpm},
\end{proof}

\begin{Th}
Let $G=G_1\times G_2$ be the direct product of two finite commutative groups, and let $A=A_1\times A_2$, where $A_1\subset G_1$, $A_2\subset G_2$. We have

\begin{equation}\label{dirdf}
\df(A)= \df (A_1)\df  (A_2),
\end{equation}

\begin{equation}\label{dird}
\delta (A_1)\delta (A_2) \leq \delta(A) \leq \delta(A_1) \df  (A_2).
\end{equation}
\end{Th}

\begin{proof}
Given sets $B_1, B_2$ with $B_1-B_1\subset A_1$,  $B_2-B_2\subset A_2$, their product $B=B_1 \times B_2$ satisfies $B-B\subset A$. Conversely, if  $B-B\subset A$,
and $B_1, B_2$ are the projections of $B$, then we have $B_1-B_1\subset A_1$,  $B_2-B_2\subset A_2$ and  $B\subset B_1 \times B_2$. This shows \eqref{dirdf}.

Given sets $B_1\subset G_1, B_2\subset G_2$ with $(B_1-B_1)\cap A_1 = \{0\} $, $(B_2-B_2)\cap A_2 = \{0\} $  their product $B=B_1 \times B_2$ satisfies  $(B-B)\cap A = \{0\} $.
This shows the lower estimate in \eqref{dird}.

To prove the upper estimate we rewrite it in the form
 \[ \frac{\Delta(A)}{q} \leq \frac{\Delta(A_1)}{q_1} \frac{1}{\Df(A_2)},   \]
where $q_i=|G_i|$ and $q=|G| =q_1q_2$. This can be rearranged as
\begin{equation}\label{direktdelta}
 \Df(A_2) \Delta(A) \leq q_2 \Delta(A_1) =\Delta(A_1 \times \{0\}).
\end{equation}
Let $B_2\subset G_2$, $B\subset G$ be maximal sets with the properties $B_2-B_2\subset A_2$, $(B-B)\cap A=\{0\}$. Then the left hand side of \eqref{direktdelta} is $|B_2| |B|$. Notice that $(\{0\}\times B_2)+B$ is a packing in $G$: if $(0,b_i)\in B_2$ and $(t_i,u_i)\in B$ (for $i=1,2$) then $(0,b_1)+(t_1,u_1)=(0,b_2)+(t_2,u_2)$ is equivalent to $(0,b_1-b_2)=(t_2-t_1, u_2-u_1)$, which is possible only if both coordinates are 0. Let $C=(\{0\}\times B_2)+B$. Then $|C|=|B_2| |B|$ due to the packing property. Also, we claim that $C-C\cap (A_1 \times \{0\})=\{(0,0)\}$. Consider  $(v_1,v_2)=(0,b_1-b_2)+(t_1-t_2, u_1-u_2)\in C-C$. Here $b_1-b_2\in A_2$ so $v_2$ can only be zero if $u_2-u_1\in A_2$, which means that $u_1-u_2\in A_2$ (recall that $A_2$ is symmetric). Also, $v_1\in A_1$ means that that $t_1-t_2\in A_1$. Therefore $(t_1-t_2,u_1-u_2)\in A_1\times A_2$, which is only possible if   $(t_1-t_2,u_1-u_2)=\{0,0\}$, and $(v_1,v_2)=\{(0,0)\}$.
\end{proof}

\begin{Ex}
  Let $G_1=G_2$, $A_1\subset G_1$ arbitrary, $A_2$ its standard complement, $A=A_1\times A_2 \subset G=G_1 \times G_2 $, $|G| =q=q_1^2$. We have
   \[ \delta(A) = \lambda(A) = 1/q_1 = q^{-1/2} . \]
   Indeed, $ \delta(A)\leq \lambda(A)=\lambda (A_1) \lambda (A_2) = |G_1|^{-1} = 1/q_1$ by the previous theorem and duality. We also have $\delta(A) \geq1/q_1  $,
since the diagonal $B = \{ (x,x): x\in G_1 \} $ satisfies $(B-B)\cap A = \{0\} $.

This is also an example when the upper estimate of \eqref{dird} holds with equality, since $\delta(A_1)\df(A_2)=1/q_1$ by duality.

In contrast, $\overline \delta(A) = \overline \delta(A_1)\overline \delta(A_2)$ can be quite near 1. A random set satisfies
 \[ \max (    \overline \Delta(A_1), \overline \Delta(A_2)  \lesssim (\log q)^2 ,    \]
 see the next section, and then we have $\overline \delta(A) \gtrsim ( \log q)^{-4}$.
\end{Ex}

\section{Random sets}\label{secrandom}

First we describe our notion of a random standard set. Given a finite group $G$, write
 \[ G_1 = \{ x\in G: 2x=0  \} ,\]
 the set of elements of order 2 (and the unit). The set $G \setminus G_1$ is a disjoint union of pairs $\{x,-x \} $; let $G_2$ be a set
 containing exactly one element of each pair. We have
  \[ G = G_1 \cup G_2 \cup -G_2, \]
  a disjoint union. Write $|G_i| =q_i$, so that $q=q_1+2q_2$.

  Take a real number $\rho\in(0,1)$. Let  $\{ \xi_y, y\in G_1 \cup G_2 \} $  be a collection of independent 0-1 valued random variable satisfying
   \[ \Pr (\xi_y=1) = \rho .\]
   Our random standard set corresponding to the prescribed probability $\rho$ will be
    \[ R = \{0 \} \cup  \{y\in G_1: \xi_y=1 \} \cup \bigcup_{y\in G_2, \xi_y=1} \{y, -y \}. \]
    Nothing depends on the value of $\xi_0$ as 0 must be in $R$ deterministically, but some formulas will look nicer using it.
    Observe that
     \[ \E  ( |R| ) = 1+\rho(q-1) . \]

The standard complement of a random set will be a random standard set corresponding to the probability $1-\rho$. In the case $\rho=1/2$
this observation, together with the dualities of Section \ref{secduality} shows that the medians of $\lambda$ and $\lambda^\pm $ are both $q^{-1/2}$.

To control various quantities related to our random set we need a large deviation estimate. Many forms of Bernstein's
(or Chernov's) inequality will work; we quote one from Tao and Vu's book \cite{taovu06}[Theorem 1.8] which is comfortable for us.

\begin{Lemma} \label{largedev}
  Let $X_1, \ldots, X_n$ be independent random variables satisfying $| X_i - \E(X_i)| \leq 1 $ for all $i$. Put
  $ X = X_1 + \ldots + X_n$ and let $\sigma^2$ be the variance of $X$. For any $t>0$ we have
   \[ \Pr( |X - \E(X)| \geq t\sigma) \leq 2 \max \left( e^{-t^2/4}, e^{-t\sigma/2} \right) .\]
\end{Lemma}

\begin{Th} \label{randomlambda}
  Assume
 $$ 1 <c < \frac{q}{32\log q}$$
(hence implicitely $q \geq 164$) and
   \[ 16c \frac{\log q}{q} < \rho < 1 - 16c \frac{\log q}{q}  . \]
   With probability exceeding $1- 2 q^{1-c}$ the random set $R$ corresponding to probability $\rho$ satisfies
    \[ \bigl| |R|-\rho q \bigr| < 3  \sqrt{c\rho(1-\rho)q \log q} ,  \]
    \[ \frac{1}{3 \sqrt{c \log q}} \sqrt{ \frac{ 1-\rho}{\rho q}} < \lambda^- (R) \leq \lambda^+(R) < {3 \sqrt{c \log q}} \sqrt{ \frac{ 1-\rho}{\rho q}}.   \]

\end{Th}

\begin{proof}
Put $ f_0(x) = \xi_x$ if $x\in G_1\cup G_2$, $f_0(x)=\xi_{-x}$ if $x\in-G_2$.
The function testifying the upper estimate will be this with a modified value at 0.

 We calculate the expectation and variance of $\hat f_0$. Clearly
  \[ \hat f_0(\gamma) = \sum_{y\in G_1} \xi_y \gamma(y) + 2 \sum_{y\in G_2} \xi_y \re \gamma(y) ,  \]
  hence
  \[ \E (\hat f_0(\gamma)) = \rho \sum_{y\in G_1} \gamma(y) + 2 \rho \sum_{y\in G_2}  \re \gamma(y) =
  \begin{cases}
    \rho q & \text{if } \gamma=\1, \\ 0 & \text{otherwise.}
  \end{cases} \]
  Similaly, the variance is

  \begin{align*}
    \D^2 (\hat f_0(\gamma)) & = \rho(1-\rho) \left( \sum_{y\in G_1} \gamma(y)^2  +  \sum_{y\in G_2}  \bigl(2 \re \gamma(y)\bigr)^2 \right) \cr
& =
  \begin{cases}
   \rho(1-\rho)( 2q_2+q) & \text{if } \gamma^2=\1, \\2\rho(1-\rho) q_2 & \text{otherwise,}
  \end{cases} \end{align*}
consequently
 \[ \D^2 (\hat f_0(\gamma))< 2\rho(1-\rho)q . \]
 We apply Lemma \ref{largedev} with an obvious rescaling (the variables $2 \re \gamma(y) \xi_y$ are bounded by 2 rather than 1)
 to obtain that in the range $ t \leq 2\sqrt{\rho(1-\rho)q}$
  \[ \Pr \bigl( |  \hat f_0(\gamma)| \geq t \sqrt{\rho(1-\rho)q} \bigr) \leq 2 e^{-t^2/8}  \  \ (\gamma \neq \1 ), \]
\[ \Pr \bigl( |  \hat f_0(\1) - \rho q | \geq t \sqrt{\rho(1-\rho)q} \bigr) \leq 2 e^{-t^2/8}  . \]
We put $t= \sqrt{8c  \log q}$ (this is in accordance with $ t \leq 2\sqrt{\rho(1-\rho)q}$, as the assumptions of the theorem on $\rho$ show), so that the right hand sides above become $2q^{-c}$. Since there are altogether $q$ possible
characters $\gamma$, with probability  $1- 2 q^{1-c}$ none of the above events happens. In this favourable case we write
\[ a =  t \sqrt{\rho(1-\rho)q} =   \sqrt{8c\rho(1-\rho)q \log q}        , \]
 \[ f(x) = \begin{cases}
   f_0(x)+a & \text{if } x=0, \\ f_0(x) & \text{otherwise,}
 \end{cases} \]
  \[ \hat f(\gamma) = \hat f_0(\gamma) + a \geq \begin{cases}
   0  & \text{always }, \\ \rho q & \text{if } \gamma=\1.
 \end{cases} \]
 This shows $f\in\mathcal S^+(R)$ and consequently
  \[ \lambda^+(R) \leq \frac{f(0)}{\hat f(\1)} < \frac{1+a}{\rho q} < {3 \sqrt{c \log q}} \sqrt{ \frac{ 1-\rho}{\rho q}}. \]

    To prove the lower estimate let $R'$ be the standard complement of $R$, which is a random standard set
for probability $1-\rho$, hence the above argument gives
  \[ \lambda^+(R')  < {3 \sqrt{c \log q}} \sqrt{ \frac{ \rho}{(1-\rho)q}} \]
    with the same probability. The lower estimate follows from the duality relation in Theorem \ref{duality}.

    The estimate of $|R|$ follows from $|R| = \hat f_0(\1)$ or $\hat f_0(\1) + 1$.
\end{proof}

Our lower and upper estimates differ by a factor of $\log q$. We have no guess whether this is necessary, or
the values of the $\lambda$'s are more concentrated. The large deviation estimate used is quite sharp. If the values of
$\hat f_0(\gamma)$ were independent for different characters $\gamma$, one could deduce that
 \[ \min  \hat f_0(\gamma) < - c_1a \]
 with high probability, with some positive constant $c_1$.
They are far from independent, but still it is likely that their dependence is not very strong, and
 the existence of large negative values can be proved.
On the other hand there is no reason to think that the uniform weights used in the proof above are near optimal.

Now we turn to estimating the $\delta$ quantities. This problem drew some attention in the case $\rho=1/2$, in the context of
estimating the clique number of Cayley graphs. Alon and Orilitsky \cite{AO} proved that typically $\Delta(R) \lesssim (\log q)^2$ in this
case. Below we adapt their proof for general $\rho$. Green \cite{Green} improved this estimate to the optimal $ O(\log q)$ for
cyclic groups. (Green considers sumsets rather than difference sets, but an adaptation to differences is
possible.) Prakash \cite{Prakash} improved  Alon and Orilitsky's estimate for general commutative groups with cardinality composed
of few primes. It is likely that Green's and Prakash' methods can also be extended to general $\rho$.

\begin{Th} \label{randomdelta}

(a)  Assume
   \[ q^{-1/2} < \rho < 1 -  q^{-1/3} {\log q}  . \]
   With probability exceeding $1- \exp \left(-c_1  \log^2 q/ \log \frac{1}{\rho} \right) $ the random set $R$ corresponding to probability $\rho$ satisfies
\begin{equation}\label{kombibecsles}
 \Df (R) < c_2 \left( \frac{\log q}{\log \frac{1}{\rho}} \right)^2, \  \overline \delta (R) > \frac{1}{c_2} \left( \frac{\log \frac{1}{\rho}}{\log q} \right)^2  . \end{equation}
    Here $c_1, c_2$ are absolute constants.
    In the range
    \[  1 -  q^{-1/3} {\log q} <\rho<  1 - 16c \frac{\log q}{q}, \  1 <c < \frac{q}{32\log q}   \]
     with probability exceeding $1- 2 q^{1-c}$ we have

      \[  \Df (R) < 3 \sqrt{c \log q}  \sqrt{ \frac{\rho q}{ 1-\rho}}, \  \overline \delta (R) >\frac{1}{3 \sqrt{c \log q}} \sqrt{ \frac{ 1-\rho}{\rho q}}. \]

(b)  Assume
   \[   q^{-1/3} {\log q}  < \rho < 1 -  q^{-1/2} . \]
   With probability exceeding $1- \exp \left(-c_1  \log^2 q/ \log \frac{1}{1-\rho} \right) $
the random set $R$ corresponding to probability $\rho$ satisfies
\begin{equation}\label{kombibecslesdual}
 \Delta(R) < c_2 \left( \frac{\log q}{\log \frac{1}{1-\rho}} \right)^2, \   \delta (R) < \frac{c_2}q  \left( \frac{\log q}{\log \frac{1}{1-\rho}} \right)^2  . \end{equation}
    Here $c_1, c_2$ are the same constants.
    In the range
    \[  16c \frac{\log q}{q} <\rho< q^{-1/3} {\log q} , \  1 <c < \frac{q}{32\log q}   \]
     with probability exceeding $1- 2 q^{1-c}$ we have
      \[  \Delta (R) < 3 \sqrt{c \log q}  \sqrt{ \frac{(1-\rho) q}{ \rho}}, \  \delta (R) < 3 \sqrt{c \log q}  \sqrt{ \frac{(1-\rho) }{ \rho q}} . \]
\end{Th}

For small values of $\rho$  estimate \eqref{randomdelta}
   stops improving; we shall study later the passage of $\overline \Delta$ from 2 to 3. For $\rho$ very near 1
the estimate becomes trivial.

We start with some preparation.
     We define the \emph{effective cardinality} of a standard set by the formula
      \[ |A| ' = |A\cap(G_1 \cup G_2)|-1 . \]
      This quantity is between $( |A|-1)/2$ and  $ |A|-1$. The probability that a difference set of a given set $B$ is contained
      in a random standard set is
       \[ \Pr (B-B \subset R) = \rho^{|B-B| '} . \]
       Consequently the expected number of difference sets of sets of cardinality $k$ contained in $R$ is
        \[ \sum_{B\subset G, |B| =k}  \rho^{|B-B| '} . \]
        This quantity is difficult to control, because we do not know enough about the distribution of $|B-B|$. When $k$ is small
compared to $q$, we expect that for most sets  $|B-B|$ will be of size $ > ck^2$, but there is no applicable result of this kind.
Instead we will select such subsets of an arbitrary set.

\begin{Lemma}\label{semisidon}
 Let $A$ be a finite set in a commutative group, $|A| = m$, and let $k$ be an integer, $1 \leq k \leq \sqrt{m}$.
  There is a $B\subset A$, $|B| = k$ satisfying
\begin{equation}\label{majdnemsidon}
    |B-B| \geq 1 + \frac{k(k-1)}2 \left( 1- \frac{k(k-1)}{2m} \right) . \end{equation}
\end{Lemma}

This lemma is also in Alon and Orilitsky's paper; below we give a slightly simpler proof.

\begin{proof}
  We use induction on $k$. Assume we found a $k$-element subset
   \[ B = \{ b_1, \ldots, b_k \} .\]
   We try to add a further element $a\in A$. The elements $a-b_i$ will be in the difference set of the  set $B' = B \cup \{a\} $; let $z_a$
   be the number of those that are already contained in $B-B$. This quantity does not exceed the number of solutions of
    \[ a - b_i = b_u-b_v, \ 1 \leq i,u,v \leq k \]
    (it may be smaller, as several pairs $u,v$ may exist for a given $i$). Hence
     \[ \sum_{a\in A} z_a \leq k^3 , \]
     consequently there is an $a\in A$ with $z_a \leq k^3/m$. This means that at least $k-k^3/m$ new differences occur, and this provides the
 inductive step.
\end{proof}

By a theorem of Koml\'os, Sulyok, Szemer\'edi \cite{kosusz}, in $\setZ_q$ we can find a set $B\subset A$ of size $|B|>c \sqrt{m}$ which
is a Sidon set, that is, all differences are distinct. In general groups we could only show the analogous result with
 $|B|>c \sqrt[3]{m}$; however, the weaker property given in Lemma \ref{semisidon} is equally applicable for our aims.

 \begin{proof}[Proof of Theorem \ref{randomdelta}.]
We are going to estimate $\Pr \bigl(\Df(R) \geq m\bigr)$. Set $k=[\sqrt{m}]$. By the lemma above, the event
$ \Df(R) \geq m$ is contained in the event
 \[ \exists B :  B-B \subset R, \ |B| =k, \ B  \text{ satisfies } \eqref{majdnemsidon} . \]
 Since \eqref{majdnemsidon} implies
  \[  |B-B|' \geq \frac{|B-B|-1}2 \geq c_4 m \]
with a suitable positive constant $c_4$, for a given $B$ the probability is $ \leq \rho^{c_4m} $. Since the number of $k$-element sets is
less than $q^k$, we obtain
 \[ \Pr \bigl(\Df(R) \geq m\bigr) < q^{\sqrt{m}}   \rho^{c_4m} .\]
This immediately gives the estimate in \eqref{kombibecsles}. The validity of this estimate is not restricted to the
range given in Theorem \ref{randomdelta}; however, for $\rho$ near to 1 we get a better result by applying Theorem \ref{randomlambda}
and the inequality $\df(R) \geq \lambda^+(R)$. This is presented in  the next formula.

This proves part (a); part (b) is the dual formulation.
 \end{proof}

 \begin{Rem}
   One can give a lower estimate for $\Df(R)$ as follows. Select sets $B_1, \ldots, B_m$ satisfying $|B_i| = k$ and
    \[ (B_i-B_i) \cap (B_j-B_j) = \{ 0 \} \]
    whenever $i \neq j$. Then the events $B_i-B_i \subset R$ will be independent and we have
    \begin{align*}
      \Pr \bigl( \Df(R) \geq k \bigr) & \geq \Pr \bigl( B_i-B_i \subset R) \text{ for some } i \bigr) \\
      & = \prod \left( 1 - \rho^{ | B_i-B_i| '} \right) .
    \end{align*}
    To make use of this one needs to find many such $B_i$ with small difference set. This is comparably easy, if $G$ has no element
of order $<k$: we take arithmetic progressions $B_i = \{ 0, b_i, 2b_i, \ldots, (k-1)b_i \} $, and a simple greedy algorithm yields
$m \geq q/k^2 $ such sets. For $\rho=1/2$ this shows that $\Df(R) \gtrsim \log q$ with high probability, so together with Green's bound
this shows the proper order of magnitude for certain groups. For general groups  a weaker form of this
argument  gives  $\Df(R) \gtrsim  \sqrt{ \log q }$.

 \end{Rem}

We now study the threshold as $\Df$ passes from 2 to 3. Elements of order 3 play a special role here. Assume $x$ is an element of
order 3. The difference set of the 3-element set (subgroup) $\{ 0, x , -x \}$ is itself, hence $\Df(A)<3$ is possible only if elements
of order 3 are all absent from $A$. To avoid this we assume that $3\nmid q$, that is, there are no elements of order 3. With some
extra effort the next result can be extended (with a properly modified notion of a random set) to all groups, save those isomorphic
to $\setZ_3^k$.

     \begin{Th} \label{kettoharom}
 Let $G$ be a finite commutative group, $|G|=q$, and assume that $3 \nmid  q$. For $\frac{6}{5}q^{-1}<\rho<q^{-2/3} $ the random set
 $R$ corresponding to probability $\rho$ satisfies
  \[ \Pr(\Df(R) \leq 2) > 1 - q^2\rho^3 .\]
     \end{Th}

     \begin{proof}
       It is easy to see that the property $\Df(R) \geq 3$ is equivalent to the existence of $a,b,c\in R$, all different from 0, such that $a+b+c=0$. For a given
       $a,b,c\in G$ we have
        \[ \Pr(a,b,c\in R) =
        \begin{cases}
          \rho^3 & \text{if they are all distinct,} \\ \rho^2 & \text{if two coincide .}
        \end{cases} \]
        (All three cannot coincide by the absence of elements of order 3, and one cannot coincide with the negative of another.)
The number of such triples $a,b,c$ containing distinct elements is $<q^2$, order counted, so without ordering it is $<q^2/6$; the number of triples containing two identical elements (that is, $a,a,-2a$) is exactly $q-1$. We obtain
\[ \Pr(\Df(R) \geq 3) < q^2\rho^3/6 + q\rho^2 < q^2 \rho^3 . \]
     \end{proof}
     \begin{Rem}
       If $\Df(R) \leq 2$, its value can be 1 or 2. The probability that it is 1 is exactly $(1-\rho)^{q_1+q_2-1}$; it becomes negligible
around $\rho \sim (\log q)/q$.

With some effort the above theorem could be complemented by an upper estimate showing that $\Pr(\Df(R) \geq 3) \to1$ if $\rho q^{2/3} \to \infty$.
     \end{Rem}

Part (a) of Theorem \ref{norelations} follows from the results of this section. Indeed, if $\rho=q^{-2/3}/2$, then the corresponding
random set satisfies $\df(R)=1/2$ and $\lambda^+(R) < c q^{-1/6} ( \log q)^{1/2} $ with positive probability, according to Theorems
\ref{kettoharom} and \ref{randomlambda}.

\section{Balls in dyadic groups}\label{secdyadic}

In this section we will prove part (b) of Theorem \ref{norelations} by studying some sets in the group $G = \setZ_2^n$ (so now $q=2^n$). The elements will be written as 0-1 sequences. For an $x\in G$
by its \emph{norm} we mean the number of coordinates equal to 1, denoted by $ \|x\|$. We consider the \emph{ball}
 \[ B_k = \{ x\in G: \|x\| \leq k \} ,\]
 and its standard complement, the \emph{antiball}
 \[ A_k = \{ x\in G: \|x\| > k \} \cup \{0\} . \]

 The size of maximal difference sets contained in $B_k$ is known: for even $k<n$ we have
\begin{equation}\label{kleitman}
\overline \Delta(B_k) = \Delta(A_k) = |B_{k/2}| = \sum_{i\leq k/2} \binom{n}{i} ,
\end{equation}
see Kleitman \cite{kleitman66}.
Much less is known about $\Delta(B_k)$, in spite of much attention, due to its interpretation as the maximal size of a set of
error-detecting codes. In this context the inequality $\delta (B_k)\leq \overline \delta (B_k)$ is known as the Hamming bound, while Delsarte \cite{delsarte73} introduced the improved bound $\delta(B_k)\leq \lambda^-(B_k)$. Asymptotically, as $k/n\to \gamma$ for some $0< \gamma <1$, the best current upper estimate  for
$\lambda^-(B_k)$ is by McEliece et al. \cite{mcelieceetal}, and numerical results in \cite{BJ} suggest this estimate actually gives the correct value of $\lambda^-(B_k)$. The best lower bound for $\delta(B_k)$ is the Gilbert-Varshamov bound given by the usual covering argument (see \cite{Sloane}). Samorodnitsky \cite{samorodnitsky01} proved that the Delsarte bound cannot match the Gilbert-Varshamov bound.

\medskip

In the sequel we apply Samorodnitsky's method from \cite{samorodnitsky98} to estimate certain $\lambda$'s of the sets
$B_k$ and $A_k$. We focus on the case $k>n/2$, which is uninteresting from the point of view of coding theory. Samorodnitsky's aspect is rather different from ours, so we repeat a part of the argument in our words. The central
ingredient is the following inequality, which is Lemma 3.3 in \cite{samorodnitsky98}.

\begin{Lemma}
  Let $F$ be a polynomial of degree at most $k$, satisfying $F(0)=1$ and $F(i) \geq 0$ for integer values of $i$, $0 \leq i \leq n$.
Assume $k\leq n$ and write $\alpha=k/(2n)$. We have
\begin{equation}\label{samorodni}
\sum_{i=0}^n \binom ni F(i) \geq c_1 n^{-1/4} {\binom{2n}{k}}^{-1/2} 2^n \geq c_2 \alpha^{1/4} \left( 2 \alpha^\alpha (1-\alpha)^{1-\alpha} \right)^n
\end{equation}
with positive absolute constants $c_1, c_2$.
\end{Lemma}

\begin{Th}
  Assume $k\leq n$ and write $\alpha=k/(2n)$,
   \[ \beta = - ( \alpha \log_2 \alpha + (1-\alpha) \log_2 (1-\alpha)) .\]
We have
\begin{equation}\label{blambda}
  \lambda(B_k) \geq c_2 \alpha^{1/4} \left(  \alpha^\alpha (1-\alpha)^{1-\alpha} \right)^n =  c_2 \alpha^{1/4} q^{-\beta},
\end{equation}

\begin{equation}\label{alambda}
  \lambda(A_k) \leq c_3 \alpha^{-1/4} q^{\beta-1},
\end{equation}
with positive absolute constants $c_2, c_3$.
\end{Th}

\begin{proof}
We want to estimate $f(0)/\hat f(\1)$ for functions $f\in \mathcal S(B_k)$ such that $\hat f\geq 0$. By Proposition \ref{automorph2} we may assume that $f$
is invariant under automorphisms that leave $B_k$ fixed. Permutations of coordinates are such automorphisms, hence $f$ depends only
on the number of coordinates equal to 1. This means that there are real numbers $a_0, \ldots, a_k$ such that $f(x)=a_i$ if $\|x\|=i\leq k$,
and  $f(x)=0$ if $\|x\| > k$. Consequently
\begin{equation}\label{permut}
\hat f(\gamma) = \sum_{i=0}^k a_i \sum_{\|x\|=i} \gamma(x) .\end{equation}
The characters of $G$ are easily described in the form
\[ \gamma_y(x) = (-1)^{\sk{x,y} }, \ y\in G \]
where $\sk{x,y}$ is the scalar product in the usual sense, so it is an integer between 0 and $n$.
This defines a natural norm for characters; we write $\|\gamma\|=\|y\|$ if $\gamma=\gamma_y$.

It is easily seen, by
  grouping the elements $x\in G$ according to the value of $j=\sk{x,y}$ that whenever $\|y\|=m$, we have
   \[     \sum_{\|x\|=i} \gamma_y(x) = \sum_{j=0}^{\min (i,m)} (-1)^j \binom{m}{j} \binom{n-m}{i-j} . \]
   The important point is that this is a polynomial of degree $i$ in $m$ (these are called Krawchouk polynomials).
   By substituting this into \eqref{permut} we obtain that
    \[ \hat f(\gamma) = F( \|\gamma\| ),\]
    where $F$ is a polynomial of degree at most $k$. We have
    \[ \hat f(\1) = F( 0 ) \]
    and, by Fourier inversion,
     \[ f(0) = \frac{1}{q} \sum_\gamma  \hat f(\gamma) = \frac{1}{q} \sum_\gamma  F( \|\gamma\| ) = \frac{1}{q} \sum_{m=0}^n \binom nm F(m) . \]
     Inequality \eqref{blambda} now follows by applying \eqref{samorodni}, and inequality \eqref{alambda} by duality (Theorem \ref{duality}).
\end{proof}

So far we did not succeed in finding a function that would constructively demonstrate inequality \eqref{alambda}.

We complement these inequalities by some easy bounds for $\lambda^\pm$.

\begin{Th}
  Assume $n/2-1<k\leq n$.

We have
\begin{equation}\label{blambdapm}
  \lambda^\pm (B_k) \leq \frac{2k+2}{q(2k+2-n)},
\end{equation}

\begin{equation}\label{alambdapm}
  \lambda^\pm (A_k) \geq  1 - \frac{n}{2k+2} .
\end{equation}
\end{Th}

\begin{proof}
  Consider the characters, corresponding to the basis vectors (with some abuse of notation):
   \[ \gamma_j(x_1, \ldots, x_n) = (-1)^{x_j} = 1 - 2x_j .\]
   Clearly
    \[ \sum \gamma_j(x) = n - 2 \|x\|,\]
    hence the function
     \[ f(x) = 2k+2-n + \sum \gamma_j(x) = 2(k+1- \|x\| ) \]
     satisfies
      \[ f \in \iS^\pm(B_k), \ f(0)= 2k+2, \ \hat f(\1) = q(2k+2-n)  . \]
      This shows \eqref{blambdapm}, and \eqref{alambdapm} follows by duality (Theorem \ref{duality}).
\end{proof}

Let us summarize the results for the set $A_k$ in the case when $\frac{1}{4} < \alpha =\frac{k}{2n} < \frac{1}{2}$. By equation \eqref{kleitman} and standard approximations for the binomial coefficients we have $\delta(A_k)=q^{\beta-1+o(1)}$. Equation \eqref{alambda} shows that $\lambda(A_k)$ is in the same range $\lambda(A_k)=q^{\beta-1+o(1)}$. On the other hand, equation \eqref{alambdapm} shows that $\lambda^{\pm}(A_k)\geq 1-\frac{1}{4\alpha}$. If $\alpha \approx 1/2$ this proves part (b) of Theorem \ref{norelations}.

\begin{Ex}\label{exlambdapm}
  We show how examples of $\lambda^- < \lambda^\pm $ are related to monotonicity of $\lambda^\pm $. Let $A$ be a set such that
  $\lambda^-(A) < \lambda^\pm (A) $, e.g. the antiball $A_k$ above. Take an $f\in \iS^- (A)$ which produces the value of  $\lambda^-(A)$, and put
   \[ A^+ = \{ x: f(x) > 0 \} . \]
   We have clearly $A^+\subset A$ and $f\in\iS^\pm (A^+) $, hence
   $$ \lambda^\pm (A^+) \leq f(0)/\hat f(\1) = \lambda^-(A) < \lambda^\pm (A). $$
\end{Ex}


\end{document}